\documentclass[reqno,a4paper,11pt]{article}
\usepackage{amsmath,amssymb,amsthm,amsfonts,mathrsfs}
\usepackage{epsf}

\theoremstyle{plain}
\newtheorem{theorem}{Theorem} [section]
\newtheorem{corollary}[theorem]{Corollary}
\newtheorem{lemma}[theorem]{Lemma}

\theoremstyle{definition}

\theoremstyle{remark}

\numberwithin{theorem}{section}
\numberwithin{equation}{section}
\numberwithin{figure}{section}

\def\R{\mathbb R}

\def\d{\delta}
\def\e{\varepsilon}

\def\var{\varphi}

\renewcommand{\R}{\mathbb{R}}
\newcommand{\one}{\mathbf{1}}
\newcommand{\ee}{\mathbf{e}}
\renewcommand{\S}{\mathbb{S}}

\renewcommand{\H}{\mathcal H}

\title{How to recognize convexity of a set from its marginals}

\author{Alessio Figalli and David Jerison}
\date{}

\begin{document}
\maketitle

\begin{abstract}
We investigate the regularity of the marginals onto
hyperplanes for sets of finite perimeter. We prove, in particular,  that 
if a set of finite perimeter has log-concave marginals onto a.e. hyperplane then 
the set is convex.
\end{abstract}

\section{Introduction}

Given $E\subset \R^n$ a Borel set,  it is well-known that if $E$ is convex then its marginals onto any hyperplane are log-concave.
More precisely, let us denote by $\one_E$ the characteristic function of $E$ (that is $\one_E(x)=1$ if $x \in E$, $\one_E(x)=0$ if $x \not\in E$),
and for any direction $\ee \in \S^{n-1}$ let $\pi_\ee:\R^n \to \ee^\perp$ denote the orthogonal projection onto the hyperplane $\ee^\perp:=\{x \in \R^n:\ee\cdot x=0\}$.
If we define
$$
w_\ee:\ee^\perp \to \R,\qquad w_\ee(x):=\int_\R \one_E(x+t\ee)\,dt,
$$
then $w_\ee$ is of the form $e^{-V}$ for some convex function $V:\ee^\perp\simeq \R^{n-1} \to \R\cup\{+\infty\}$.
Actually, by the Brunn-Minkowski inequality, an even stronger result is true, namely $w_\ee^{1/(n-1)}$ is concave.
(We refer to \cite{G} for more details.)

The aim of this paper is to show that, under rather weak regularity assumptions on $E$,
the converse of this result is true: if a set has log-concave marginals onto a.e. hyperplane $\ee^\perp$ then it is convex.
Actually, we will prove a stronger result: we will not assume that the marginals are log-concave, but only that they have convex support and are uniformly
Lipschitz strictly inside their support.\\

To state our result, let us introduce some notation.
For any direction $\ee$, we define the set $A_\ee:=\{w_\ee>0\}\subset \ee^\perp$.
(Notice that if $w_\ee$ is log-concave then $A_\ee$ is convex.)
Also, for any $\delta>0$ we set $A_\ee^\delta:=\{x \in A_\ee:{\rm dist}(x,\partial A_\ee)\geq \delta\}$.
We recall that a set $E$ is of \textit{finite perimeter} if the distributional derivative $\nabla \one_E$ of $\one_E$ is a finite measure, that is
$$
\int_{\R^n} |\nabla \one_E| <\infty.
$$
Also, we use $\H^{k}$ to denote the $k$-dimensional Hausdorff measure.
Here is our main result:
\begin{theorem}
\label{thm:main}
Let $E\subset \R^n$ be a bounded set of finite perimeter and assume that $A_\ee$ is convex for $\H^{n-1}$-a.e. $\ee \in \S^{n-1}$.
Suppose further that $w_\ee$ is locally Lipschitz inside $A_\ee$ for $\H^{n-1}$-a.e. $\ee \in \S^{n-1}$ and the following uniform bound holds: for any $\delta>0$ there exists a constant $C_\delta$ such that
$$
|\nabla w_\ee| \leq C_\delta \quad \text{a.e. inside $A_\ee^\delta$}, \qquad \text{for $\H^{n-1}$-a.e. $\ee \in \S^{n-1}$.}
$$
Then $E$ is convex (up to a set of measure zero).
\end{theorem}

Notice that, since log-concave functions are uniformly Lipschitz in the interior of their support, our assumption is weaker than asking that $w_\ee$ is log-concave for $\H^{n-1}$-a.e. $\ee \in \S^{n-1}$.
Hence our theorem implies the following:
\begin{corollary}
Let $E\subset \R^n$ be a bounded set of finite perimeter and assume that, for $\H^{n-1}$-a.e. $\ee \in \S^{n-1}$, $w_\ee$ is log-concave.
Then $E$ is convex (up to a set of measure zero).
\end{corollary}
In light of the fact that marginals of log-concave functions are log-concave (see for instance \cite[Section 11]{G}),
one may wonder if the corollary above may be generalized to functions, that is, whether the fact that $\var:\R^n\to [0,+\infty)$ has 
log-concave marginals implies that $\var$ is log-concave.
Unfortunately this stronger result is false. To see this, consider $\var:=\one_{B_1}-\e \psi$, where $\psi:\R^n\to \R$ is a smooth radial non-negative function supported in a small neighborhood of the origin.
Since the marginals of $\one_{B_1}$ are positive and uniformly log-concave near the origin, it is easy to see that the marginals of $\var$ are log-concave for $\e>0$ sufficiently small,
but of course $\var$ is not log-concave.\\

The assumption that $E$ is of finite perimeter is technical, and we expect the result to be true without this assumption.   However, finite perimeter plays an important role 
in the proof, which is based on measuring the perimeter using a fractional
Sobolev norm of a smoothing of $\one_E$.   To formulate that result,
we introduce some further notations. 

For notational convenience, we will say that $a \lesssim b$ if there exists a dimensional constant $C$ such that $a \leq Cb$, and $a \simeq b$ if $a \lesssim b$ and $b \lesssim a$.  

Consider the $1/2$ Sobolev norm, defined by
\begin{equation}
\label{eq:norm}
\|u \|_{H^{1/2}}^2:=\int_{\R^n}\int_{\R^n}\frac{|u (x) - u(y)|^2}{|x-y|^{n+1}}\,dy\,dx
\end{equation}
This norm can be used to measure the perimeter as follows:
\begin{theorem} \label{thm:norm}  Let $E\subset B_1 \subset \R^n$ have
finite perimeter.  Let $\gamma_n$ be the standard Gaussian in $\R^n$,
$\displaystyle \gamma_{n,\e}(x):=\frac{1}{\e^n}\gamma_n(x/\e)$,  and
$\var_\e:=\one_E\ast \gamma_{n,\e}$.    Then 
\[
\limsup_{\e \to 0 } \frac{1}{|\log \e|} \| \var_\e\|_{H^{1/2}}^2  \simeq
\liminf_{\e \to 0 } \frac{1}{|\log \e|} \| \var_\e\|_{H^{1/2}}^2  \simeq
\int_{\R^n} |\nabla \one_E| 
\]
i.e., the ratios of these quantities are bounded above and below
by positive dimensional constants.
\end{theorem}
 We have not investigated 
the connection, if any, between our norm and the notion of fractional perimeter 
introduced in \cite{CRS} and whose relationship to the classical perimeter can
be found in \cite{ADM,CV}.  Our norm is somewhat  different in the spirit.
On the one hand, as our norm is quadratic, the 
analysis performed in \cite{CV,ADM} does not apply in our situation.  On the
other hand, the Hilbert structure allows us to exploit Fourier transform techniques.\\

As we shall see, Theorem \ref{thm:main}  in $n$ dimensions follows easily from the case $n=2$. In two dimensions the result says  if almost every marginal is supported on an interval and is uniformly Lipschitz in its interior, then the set is convex up
to set of measure zero.
The strategy of the proof is the following. 
Consider $E\subset \R^2$.  Given $\theta \in [0,\pi]$, set $w_\theta:=w_{\ee_\theta}$ where $\ee_\theta=(\cos(\theta),\sin(\theta))$.  If $E$ 
is smoothly bounded, but not convex, then one expects
that for some direction $\theta$, the derivative of the marginal $w_\theta$ is infinite at some interior point of (the convex hull of) its support (see Figure \ref{fig1}).
\begin{figure}[ht]
\caption{{\small If $E$ is a smooth non-convex set, for a.e. $\theta$ the marginal $w_\theta$ has infinite derivative at some interior point.
However, it is easy to check that this argument fails when $E$ is not smooth (consider for instance a disk with a small square removed from its interior).
Still, we can show that some suitable integral quantity has to blow-up.}}
\label{fig1}
\end{figure}
For domains
that merely have finite perimeter, the quantity that diverges is,
roughly speaking, a suitable localized version of  $\int\int |w_\theta'(t)|^2dt\,d\theta$.   We
make this quantitative by considering the mollification $\var_\e$ of $\one_E$.  

More precisely, given a bounded set of finite perimeter $E$, and $\var_\e$ as in Theorem \ref{thm:norm} above,
we show that 
$$
\|\var_\e\|_{H^{1/2}}^2 \lesssim |\log(\e)| \qquad \text{as $\e \to 0$},
$$
and a localized version of Theorem \ref{thm:norm} saying that 
\begin{equation}
\label{eq:local norm}
\frac{1}{|\log(\e)|}\int_{B_r(x_0)}\int_{B_r(x_0)}\frac{|\var_\e(x) - \var_\e(y)|^2}{|x-y|^{n+1}}\,dy\,dx
\end{equation}
controls from above the perimeter of $E$ inside $B_{r/2}(x_0)$ as $\e \to 0$.

We then focus on the case $n=2$ and, by use of the Fourier transform, show
that \eqref{eq:local norm} is majorized by 
\begin{equation}
\label{eq:norm marginals}
\frac{1}{|\log(\e)|}\int_0^\pi \int_\R \psi(t)^2 |w_\theta' \ast \gamma_{1,\e}|^2(t)\,dt\,d\theta,
\end{equation}
where $\psi:\R\to\R$ is a suitable smooth cut-off function.
Now, under our assumption on $w_\theta$, the fact that $|\log \e| \to \infty$ 
as $\e\to 0$ shows that the measures
$$
\frac{1}{|\log(\e)|} |w_\theta' \ast \gamma_{1,\e}|^2(t)\,dt\,d\theta
$$
concentrate on the union over $\theta$ of the boundaries of the support of $w_\theta$ (which correspond to the boundary of the support of the projections of the convex
hull of $E$). If $E$ is not convex, it is not difficult to see that this information is incompatible with the fact that the expression in \eqref{eq:local norm}
controls from above the perimeter of $E$ at every point (in particular, at points on the reduced boundary of $E$ which are inside the support of the convex hull), proving the result.
Once the theorem is proved in two dimensions, the higher dimensional case follows by a slicing argument.\\

The paper is organized as follows.  In Section \ref{sect:BV H}
we prove Theorem \ref{thm:norm} and a local version, valid in all dimensions,
showing that \eqref{eq:local norm} controls the perimeter.  
Then, in Section \ref{sect:H marginals} we majorize \eqref{eq:local norm} by
\eqref{eq:norm marginals} in dimension $2$.   
Finally, in Section \ref{sect:proof} we prove Theorem \ref{thm:main}.\\

\textit{Acknowledgements:} AF was partially supported by NSF Grant DMS-0969962.
DJ was partially supported by NSF Grant DMS-1069225.
Most of this work has been done during AF's visit at MIT in the Fall 2012,
and during DJ's visit at UT Austin in January 2013.
Both authors acknowledge the warm hospitality of these institutions.

\section{A Hilbert norm for sets of finite perimeter}
\label{sect:BV H}

In this section we prove Theorem \ref{thm:norm}.
This argument is valid in any dimension. 
Let us recall that  $u \in BV(\R^n)$ if its distributional derivative $\nabla u$ is a finite measure, and 
\[
\| u \|_{BV} = \int_{\R^n} |\nabla u| = |\nabla u|(\R^n).
\]
Given a set $E$ of finite perimeter, there is a suitable notion of boundary, 
called the \textit{reduced boundary} and denoted by $\partial^*E$, such that the following is true:
$$
\int_{\R^n}|\nabla \one_E|=\H^{n-1}(\partial^*E),
$$
and for any $x \in \partial^*E$ the following hold:
\begin{equation}\label{eq:fullmeasure}
\H^{n-1}(\partial^*E \cap B_r(x))\simeq r^{n-1}\qquad \text{as $r\to 0$}
\end{equation}
and there exists a unit vector $\nu(x)$ (called \textit{inner measure theoretical normal to $E$ at $x$})
such that
\begin{equation}
\label{eq:blow up}
\frac{1}{|B_r|}\int_{B_r}|\one_E(x+y)-\one_{\R^+}(\nu(x)\cdot y)|\,dy\to 0 
\qquad \text{as $r\to 0$}
\end{equation}
Furthermore, $\nu$ is a measurable function of $x\in \partial^*E$.  
(We refer to \cite[Sections 3.3 and 3.5]{AFP} or \cite[Chapters 12 and 15]{M} for more details.)

\begin{lemma}
\label{lem:1}
Let $u \in BV(\R^n)$ be supported in the unit ball $B_1$.  
Assume $|u|\leq 1$, and set $u_\e:=u\ast \gamma_{n,\e}$. Then there is 
a dimensional constant $C$ such that
$$
\|u_\e\|_{H^{1/2}}^2 \leq C |\log(\e)|\, \int_{\R^n} |\nabla u|, \qquad 0 < \e < 1/2
$$
\end{lemma}
\begin{proof}
We begin by showing that 
\begin{equation}\label{eq:upperSobolev}
\|u_\e\|_{H^{1/2}}^2 \leq 
C \int_\e^\infty \int_{\R^n}|\partial_tu_t|^2\,dx\,dt, \qquad u_t=u\ast \gamma_{n,t}.
\end{equation}
Indeed, recall that 
\[
\|u_\e\|_{H^{1/2}}^2 = c \int_{\R^n} |\xi||\hat u_\e(\xi)|^2 \, d\xi 
\]
Furthermore,
\[
\partial_t \hat u_{\e+t}(\xi) = \partial_t \bigl(e^{-(t+\e)^2|\xi|^2}\bigr) \hat u(\xi) 
= -2(t+\e)|\xi|^2 e^{-(t^2+2\e t)|\xi|^2} \hat u_\e(\xi),
\]
thus, using the fact that the Fourier transform is an isometry in $L^2$, up to a multiplicative constant we get
\begin{align*}
\int_\e^\infty \int_{\R^n}|\partial_tu_t|^2\,dx\,dt =
\int_0^\infty \int_{\R^n} (t+\e)^2 |\xi|^4 e^{-2(t^2+2\e t)|\xi|^2} 
|\hat u_\e(\xi)|^2 \, d\xi \, dt .
\end{align*}
Moreover,
\begin{align*}
\int_0^\infty  (t+\e)^2 |\xi|^4 e^{-2(t^2+2\e t)|\xi|^2}\, dt 
& \geq 
\int_\e^\infty  t^2 |\xi|^4 e^{-8t^2|\xi|^2} \, dt 
+ 
\int_0^\e \e^2 |\xi|^4 e^{-8\e t|\xi|^2} \, dt  \\
& =
\int_{\e |\xi|}^\infty u^2 |\xi|^2 e^{-8u^2}\, \frac{du}{|\xi|}
+ 
\int_0^{\e^2 |\xi|^2} \e^2|\xi|^4 e^{-8u^2}\, \frac{du}{\e |\xi|^2}  \\
& \ge c|\xi|
\end{align*}
(If $\e |\xi| \le 1$, the first integral majorizes $|\xi|$, and if $\e |\xi| \ge 1$, then
the second integral majorizes $|\xi|$.)
Therefore, \eqref{eq:upperSobolev} follows.

Next, using the formula above for $\partial_t \hat u_{1+t}(\xi)$ (i.~e., $\e = 1$)
and integrating over $t\in [1,\infty)$, we have, up to a multiplicative
constant,
\[
\int_2^\infty \int_{\R^n}|\partial_tu_t|^2\,dx\,dt
= \int_{\R^n} \int_{1}^\infty (t+1)^2 |\xi|^4 e^{-2(t^2 + 2t)|\xi|^2}\, dt \, |\hat u_{1}(\xi)|^2 \, d\xi
\]
and
\[
\int_{1}^\infty (t+1)^2 |\xi|^4 e^{-2(t^2 + 2t)|\xi|^2} \, dt 
\le \int_0^\infty 4t^2 |\xi|^4 e^{-t^2|\xi|^2} \, dt = c|\xi|.
\]
Thus, we have 
$$
\int_2^\infty \int_{\R^n}
|\partial_tu_t|^2\,dx\,dt \le C \|u\ast \gamma_{n,1}\|_{H^{1/2}}^2 \leq C \|u\|_{L^1}^2 \le C \|u\|_{L^1}
$$
for some dimensional constant $C$ (since $u$ is bounded by $1$ and supported
in the unit ball).  Finally, again using that $u$  is supported in the unit ball,
the Sobolev inequality for BV functions \cite[Chapter 3.4]{AFP} implies
\begin{equation}
\label{eq:bound 1 infty}
\int_2^\infty \int_{\R^n}|\partial_tu_t|^2\,dx\,dt\leq C \int_{\R^n} |\nabla  u| .
\end{equation}
We turn next to the integral from $\e$ to $2$.  Recall that since
$\gamma_{n,t}(x):=\frac{1}{t^n}\gamma_n(x/t)$, the time derivative of 
$u_t$ can be written as
$$
\partial_tu_t(x) = \partial_t \int u(x-tz)\gamma_n(z)\,dz
= - \int \nabla  u(x-tz)\cdot z\gamma_n(z) \, dz
$$
It follows from Fubini's theorem that 
\[
\int_{\R^n} |\partial_tu_t| \,dx\leq C \int_{\R^n} |\nabla  u| 
\]
We can also write 
\begin{align*}
\partial_tu_t(x)&= \partial_t \int u(x-z)\frac{1}{t^n}\gamma_n\biggl(\frac{z}t\biggr)\,dz
=-\frac{1}{t}\int u(x-z)\biggl[n\gamma_n\biggl(\frac{z}t\biggr)
+ \nabla \gamma_n
\biggl(\frac{z}t\biggr)\cdot \frac{z}t\biggr] \,dz\\
&= -\frac{1}{t}\int u(x-tz) \bigl[n\gamma_n(z)
+ \nabla \gamma_n(z)\cdot z\bigr]  \,dz
\end{align*}
from which, along with $|u|\le 1$,  we get
$$
 \|\partial_tu_t\|_\infty \leq \frac{C}t.
$$
Thus
\begin{equation}
\label{eq:bound2}
\int_{\R^n} |\partial_tu_t|^2 \,dx\leq \frac{C}t \int_{\R^n} |\nabla u| 
\end{equation}
Combining \eqref{eq:upperSobolev},  \eqref{eq:bound 1 infty}, and \eqref{eq:bound2}, we conclude that
\begin{align*}
\| u_\e\|_{H^{1/2}}^2 & \lesssim \int_\e^\infty\int_{\R^n} |\partial_tu_t|^2\,dx\,dt  \\
& \leq \biggl(C+ \int_\e^2 \frac{C}t \,dt\biggr)\int_{\R^n}  |\nabla u| \ \leq 
C |\log \e |\, \int_{\R^n} |\nabla u|, \qquad 0 < \e < 1/2,
\end{align*}
as desired.
\end{proof}

We now show that the norm \eqref{eq:norm} controls the perimeter locally.
\begin{lemma}
\label{lem:2}
Let $E$ be a set of finite perimeter and let $x_0 \in \partial^*E$. For
$r_0 > 0$, 
$$
\liminf_{\e\to 0}\frac{1}{|\log(\e)|}\int_{B_{r_0}(x_0)} 
\int_{B_{r_0}(x_0)} \frac{|\var_\e(x)-\var_\e (y)|^2}{|x-y|^{n+1}} \,dx\,dy 
\gtrsim 
\H^{n-1}(B_{r_0/2}(x_0)\cap \partial^* E).
$$
\end{lemma}
\begin{proof}
For $x\in \partial^*E$ define
\[
D_k(x) :=  \sup_{j\ge k}
\frac{1} { |B_{2^{-j} }| } \int_{B_{ 2^{-j} } }
|\one_E(x+y)-\one_{\R^+}(\nu(x)\cdot y)|\,dy,
\]
Let $\d>0$ be a small  dimensional constant  (to be fixed later).  
Let 
\[
F_m = \{ x\in \partial^* E: D_m(x) \le \delta\}.
\]
Then
\[
F_m \subset F_{m+1} \quad \text{and}\quad \bigcup_m F_m = \partial^* E
\]
by \eqref{eq:blow up}.   Furthermore,
$F_m$ is measurable because $\nu$ is measurable. Therefore,
we can choose $m$ sufficiently large that 
\[
\H^{n-1}(B_{r_0/2}(x_0)\cap F_m)  
\ge \frac12 \H^{n-1}(B_{r_0/2}(x_0)\cap \partial^* E)
\]
Let $\rho = 2^{-m}$, and suppose that $\sqrt{\e} < \rho/100$.    Fix 
$r\in (100\e, \sqrt\e)$.  By a standard covering argument\footnote{
If $\{B_r(x_j)\}_{1\leq j \leq N_r}$
is a maximal disjoint family of balls with
$x_j \in F_m \cap B_{r_0/2}(x_0)$, then $\bigcup_{1\leq j \leq N_r} B_{3r}(x_j)$ covers
$F_m \cap B_{r_0/2}(x_0)$, and by the definition of $\H^{n-1}$ (see \cite[Section 2.8]{AFP} or \cite[Chapter 3]{M}) we get 
$$
\H^{n-1}(F_m \cap B_{r_0/2}(x_0)) \leq C_n\sum_{j=1}^{N_r} (3r)^{n-1} =C_n3^{n-1} N_r\,r^{n-1},
$$
where $C_n>0$ is a dimensional constant.
}
there are $N = N_r$ disjoint balls $B_r(x_j)$, $j=1, \dots, N_r$,
such that $x_j\in F_m \cap B_{r_0/2}(x_0)$, and 
\[
N_r \gtrsim  \H^{n-1}(B_{r_0/2}(x_0)\cap \partial^* E)/r^{n-1}
\]
We now want to estimate from below
$$
\int_{x,y \in B_{r_0}(x_0),\,r/4\leq |x-y|\leq r/2}\frac{|\var_\e(x)-\var_\e (y)|^2}{|x-y|^{n+1}} \,dx\,dy.
$$
Since the balls $B_{r}(x_j)$ are disjoint we have
\begin{align*}
&\int_{x,y \in B_{r_0}(x_0),\,r/4\leq |x-y|\leq r/2}\frac{|\var_\e(x)-\var_\e (y)|^2}{|x-y|^{n+1}} \,dx\,dy\\
&\geq
\sum_{j=1}^N\int_{B_{r/2}(x_j)\cap \{\nu(x_j)\cdot(x-x_j)\geq r/10\}}\,dx
\int_{B_{r/2}(x_j)\cap \{\nu(x_j)\cdot(y-x_j)\leq -r/10\}\cap \{r/4\leq |x-y|\leq r/2\}}\,dy \,\frac{|\var_\e(x)-\var_\e (y)|^2}{|x-y|^{n+1}} .
\end{align*}
Because inside $B_{r/2}(x_j)$ the set $E$ is very close in $L^1$ to the hyperplane $\{\nu(x_j)\cdot(x-x_j)\geq 0\}$ (since $D_m(x_j)\leq \delta$) and $r \geq 100\e$,
we deduce that, provided $\delta$ is chosen sufficiently small (the smallness depending only on the dimension),
$|\var_\e(x)-\var_\e (y)| \gtrsim 1$ on a substantial fraction of
the latter integrals.  Thus, since $|x-y|\leq r/2$ and both $x$ and $y$ 
vary inside some sets whose measure is of order $r^{n}$, we get 
\begin{align*}
\sum_{j=1}^N \int_{B_{r/2}(x_j)\cap \{\nu(x_j)\cap(x-x_j)\geq r/10\}}\,dx &
\int_{B_{r/2}(x_j)\cap \{\nu(x_j)\cap(y-x_j)\leq -r/10\}}\,dy \,\frac{|\var_\e(x)-\var_\e (y)|^2}{|x-y|^{n+1}} \\
&\gtrsim N_r r^{n-1} \gtrsim
 \H^{n-1}(B_{r_0/2}(x_0) \cap \partial^* E )
\end{align*}
Thus, we proved that for any $r \in (100\e,\sqrt{\e})$
$$
\int_{x,y \in B_{r_0}(x_0),\,r/4\leq |x-y|\leq r/2}\frac{|\var_\e(x)-\var_\e (y)|^2}{|x-y|^{n+1}} \,dx\,dy\gtrsim  \H^{n-1}(B_{r_0/2}(x_0) \cap \partial^* E )
$$
Hence, choosing $r=4^{-k}$ and letting $k$ vary between $\ell_1:=\lfloor-\log_4(\sqrt{\e})\rfloor+1$
and $\ell_2:=\lfloor -\log_4(100\e)\rfloor$, for $\e$ sufficiently small we get 
\begin{align*}
\int_{B_{r_0}(x_0)} \int_{B_{r_0}(x_0)} \frac{|\var_\e(x)-\var_\e (y)|^2}{|x-y|^{n+1}} \,dx\,dy&\geq 
\sum_{k=\ell_1}^{\ell_2} \int_{x,y \in B_{r_0}(x_0),\,4^{-k}/4\leq |x-y|\leq 4^{-k}/2}\frac{|\var_\e(x)-\var_\e (y)|^2}{|x-y|^{n+1}} \,dx\,dy\\
&\gtrsim (\ell_2-\ell_1)\,
 \H^{n-1}(B_{r_0/2}(x_0) \cap \partial^* E )
\\
& \gtrsim |\log(\e)|\,
 \H^{n-1}(B_{r_0/2}(x_0) \cap \partial^* E )
\end{align*}
as desired.
\end{proof}

\begin{proof}[Proof of Theorem \ref{thm:norm}.]
The upper bound in Theorem \ref{thm:norm} follows from Lemma 
\ref{lem:1}, and the lower bound follows from Lemma  \ref{lem:2} letting $x_0=0$,
$r_0 =4$.
\end{proof}

\section{The $H^{1/2}$ norm expressed in terms of the marginals}
\label{sect:H marginals}
Here we prove the well known fact that the $H^{1/2}(\R^2)$ norm of a function 
is equal (up to constants) to the average of the $H^1(\R^2)$ norm of its marginals,
and then we prove a localized version of this identity.
The arguments in this section are specific to the case $n=2$.

Given a smooth rapidly decaying function $\var:\R^2 \to \R$, for any $\theta \in [0,\pi]$ we define the marginal
\begin{equation}
\label{eq:w theta}
w_\theta(t):=\int_\R \var\bigl(R_\theta(t,s)\bigr)\,ds,
\end{equation}
where $R_\theta:\R^2\to \R^2$ denotes the counterclockwise rotation by an angle $\theta$ around the origin.

\subsection{A global identity}
We claim that the norm
$$
\int_0^{\pi} \int_\R |w_\theta'|^2(t)\,dt\,d\theta
$$
is equivalent to the $H^{1/2}$ norm of $\var$.

To prove this, we first compute the Fourier transform of $w_\theta$.
We denote by $\ee_\theta:=R_\theta e_1=(\cos\theta,\sin\theta)$. Then
\begin{align*}
\hat w_\theta(\tau)&:= \int_\R w_\theta(t) e^{it\tau}\,dt=\int_{\R^2}  \var\bigl(R_\theta(t,s)\bigr)e^{it\tau}\,dt\,ds\\
&=\int_{\R^2}  \var\bigl(R_\theta x\bigr)e^{i\tau e_1\cdot x}\,dx=\int_{\R^2}\var(x)e^{i(\tau \ee_\theta)\cdot x}\,dx=\hat\var\bigl(\tau \ee_\theta\bigr),
\end{align*}
where (by abuse of notation) we used $\hat w$ and $\hat \var$ to denote
respectively the Fourier transform on $\R$
and on $\R^2$.

Thanks to the formula above and the fact that the Fourier transform is an isometry in $L^2$, we get
(up to a multiplicative constant)
\begin{align*}
\int_0^{\pi} \int_\R |w_\theta'|^2(t)\,dt\,d\theta
=\int_0^{\pi} \int_\R |\tau|^2|\hat w_\theta|^2(\tau)\,d\tau\,d\theta
=\int_0^{\pi} \int_\R |\tau|^2| \hat\var|^2\bigl(\tau \ee_\theta\bigr)\,d\tau\,d\theta.
\end{align*}
It is now easy to check that the last integral is simply an integration in polar coordinates,
so by setting $\xi:=\tau e^{i\theta}$ (so that $|\tau|=|\xi|$ and $d\xi=|\tau|\,d\tau\,d\theta$) we get
\begin{equation}
\label{eq:global norm bound}
\int_0^{\pi} \int_\R |w_\theta'|^2(t)\,dt\,d\theta= \int_{\R^2}|\xi| \,| \hat\var|^2(\xi)\,d\xi
= \bar c\int_{\R^2}\int_{\R^2} \frac{|\var(x) - \var(y)|^2}{|x-y|^3}\,dy\,dx
\end{equation}
for some dimensional constant $\bar c>0$,
which proves the claim.

\subsection{A localized identity}

Let $\psi:\R\to \R$ be a smooth compactly supported function.
By the properties of the Fourier transform we have
$$
\int_0^\pi \int_\R \psi(t)^2 |w_\theta'|^2(t)\,dt\,d\theta
=\int_0^\pi   \int_\R \biggl|\int_\R\hat\psi(\tau-\sigma ) \sigma  \hat \var(\sigma  \ee_\theta)\,d\sigma \biggr|^2\,d\tau\,d\theta.
$$
We now notice that,
since  $\hat \var(s \ee_\theta)=\int_{\R^2}\var(x)e^{-is \ee_\theta\cdot x}\,dx$,
\begin{align*}
&\biggl|\int_\R\hat\psi(\tau-\sigma ) \sigma  \hat \var(\sigma  \ee_\theta)\,d\sigma \biggr|^2\\
&=\int\int\int\int \hat\psi(\tau-\sigma ) \sigma  \var(x)e^{-i\sigma  \ee_\theta\cdot x}
\overline{\hat\psi(\tau-\upsilon )} \upsilon  \var(y)e^{i\upsilon  \ee_\theta\cdot y}\,d\sigma \,d\upsilon \,dx\,dy\\
&=\int \int \var(x)\var (y) \int \hat\psi(\tau-\sigma )\sigma e^{-i\sigma  \ee_\theta\cdot x}\,d\sigma 
\overline{\int \hat\psi(\tau-\upsilon ) \upsilon  e^{-i\upsilon  \ee_\theta\cdot y}\,d\upsilon }\,dx\,dy\\
&= \int \int \var(x)\var (y) [\ee_\theta\cdot \nabla_x] \int \hat\psi(\tau-\sigma )e^{-i\sigma  \ee_\theta\cdot x}\,d\sigma 
[\ee_\theta\cdot \nabla_y] \overline{\int \hat\psi(\tau-\upsilon ) e^{-i\upsilon  \ee_\theta\cdot y}\,d\upsilon }\,dx\,dy\\
&=\int \int \var(x)\var (y) [\ee_\theta\cdot \nabla_x] [\ee_\theta\cdot \nabla_y]
\Bigl(\psi(\ee_\theta\cdot x)\psi(\ee_\theta\cdot y) e^{-i\tau \ee_\theta\cdot (x-y)}\Bigr)\,dx\,dy.
\end{align*}
Hence, integrating this expression with respect to $\tau$ and $\theta$ we get
\begin{equation}
\label{eq:local norm t} 
\begin{split}
&\int_0^\pi \int_\R \psi(t)^2 |w_\theta'|^2(t)\,dt\,d\theta\\
&=\int_0^\pi \int_{\R^2} \int_{\R^2} \var(x)\var (y) [\ee_\theta\cdot \nabla_x] [\ee_\theta\cdot \nabla_y]
\Bigl(\psi(\ee_\theta\cdot x)\psi(\ee_\theta\cdot y) \delta\bigl(\ee_\theta\cdot (x-y)\bigr)\Bigr)\,dx\,dy\,d\theta.
\end{split}
\end{equation}
We now claim that
\begin{equation}
\label{eq:decay}
 \int_0^\pi \int_{\R^2} \int_{\R^2} \Phi(x) [\ee_\theta\cdot \nabla_x] [\ee_\theta\cdot \nabla_y]
\Bigl(\psi(\ee_\theta\cdot x)\psi(\ee_\theta\cdot y) \delta\bigl(\ee_\theta\cdot (x-y)\bigr)\Bigr)\,dx\,dy\,d\theta=0
\end{equation}
for any smooth rapidly decaying function $\Phi$.
Indeed, the  integral above is equal to the limit of 
$$
\int_0^\pi \int_{B_R} \int_{\R^2} \Phi(x) [\ee_\theta\cdot \nabla_x] [\ee_\theta\cdot \nabla_y]
\Bigl(\psi(\ee_\theta\cdot x)\psi(\ee_\theta\cdot y) \delta\bigl(\ee_\theta\cdot (x-y)\bigr)\Bigr)\,dx\,dy\,d\theta
$$
as $R\to \infty$, and the latter integral is equal to
$$
-\int_0^\pi \int_{\R^2} [\ee_\theta\cdot \nabla_x \Phi(x)] \biggl(\psi(\ee_\theta\cdot x) \int_{\partial B_R} 
\psi(\ee_\theta\cdot y) \delta\bigl(\ee_\theta\cdot (x-y)\bigr)\,\nu_{\partial B_R}(y)\cdot \ee_\theta\,d\H^{1}(y)\biggr)\,dx\,d\theta.
$$
Next, we have the majorization
$$
\Bigl|\psi(\ee_\theta\cdot y) \,\nu_{\partial B_R}(y)\cdot \ee_\theta\Bigr| \leq \frac{C}{R} \qquad \text{for $R \gg 1$}
$$
(since $\psi$ is compactly supported, so $|\nu_{\partial B_R}(y)\cdot \ee_\theta|\leq C/R$ on the support of $\psi(\ee_\theta\cdot y)$).  Furthermore, for each $\theta$,
$\psi(\ee_\theta\cdot y)$ is supported on portion of the circle $\partial B_R$ of length $O(1)$.   Thus the integral is $O(1/R)$, and the claim follows.

By exchanging the roles of $x$ and $y$, we deduce that \eqref{eq:decay} holds also if we replace $\Phi(x)$ by $\Phi(y)$.
Hence, by \eqref{eq:decay} applied with $\Phi(x)=\var(x)^2$ and $\Phi(y)=\var(y)^2$ we deduce that the expression \eqref{eq:local norm t} 
is equal to
\begin{align*}
&\int_0^\pi \int_{\R^2} \int_{\R^2} |\var(x)-\var (y)|^2 [\ee_\theta\cdot \nabla_x] [\ee_\theta\cdot \nabla_y]
\Bigl(\psi(\ee_\theta\cdot x)\psi(\ee_\theta\cdot y) \delta\bigl(\ee_\theta\cdot (x-y)\bigr)\Bigr)\,dx\,dy\,d\theta\\
&=\int_{\R^2} \int_{\R^2} |\var(x)-\var (y)|^2 
\int_0^\pi \Bigl(\psi(\ee_\theta\cdot x)\psi(\ee_\theta\cdot y) \delta''\bigl(\ee_\theta\cdot (x-y)\bigr)\Bigr)\,d\theta\,dx\,dy\\
&\qquad+\int_{\R^2} \int_{\R^2} |\var(x)-\var (y)|^2 
\int_0^\pi \Bigl(\bigl[\psi'(\ee_\theta\cdot x)\psi(\ee_\theta\cdot y)
+ \psi(\ee_\theta\cdot x)\psi'(\ee_\theta\cdot y)\bigr] \,\delta'\bigl(\ee_\theta\cdot (x-y)\bigr)\Bigr)\,d\theta
\,dx\,dy\\
&\qquad+ \int_{\R^2} \int_{\R^2} |\var(x)-\var (y)|^2 
\int_0^\pi \Bigl(\psi'(\ee_\theta\cdot x)\psi'(\ee_\theta\cdot y) \delta\bigl(\ee_\theta\cdot (x-y)\bigr)\Bigr)\,d\theta\,dx\,dy.
\end{align*}
We now observe that, by the chain-rule,
$$
\delta'\bigl(\ee_\theta\cdot (x-y)\bigr)=\frac{\partial_{\theta}\delta\bigl(\ee_\theta\cdot (x-y)\bigr)}{e_\theta^\perp\cdot (y-x)},
$$
$$
\delta''\bigl(\ee_\theta\cdot (x-y)\bigr)=\frac{\partial_{\theta\theta}\delta\bigl(\ee_\theta\cdot (x-y)\bigr)}{\bigl(e_\theta^\perp\cdot (y-x)\bigr)^2}
+ \frac{\partial_{\theta}\delta\bigl(\ee_\theta\cdot (x-y)\bigr)}{\bigl(e_\theta^\perp\cdot (y-x)\bigr)\bigl(e_\theta\cdot (y-x)\bigr)}.
$$
Hence, if we integrate by parts in $\theta$ so that no derivatives fall onto $\delta\bigl(\ee_\theta\cdot (x-y)\bigr)$, we get that there is only one term
with $\bigl(\ee_\theta\cdot (x-y)\bigr)^{-2}$, and all the others are smooth functions of $\theta$ in a neighborhood of the support of
$\delta\bigl(\ee_\theta\cdot (x-y)\bigr)$  multiplied at most by $\bigl(\ee_\theta\cdot (x-y)\bigr)^{-1}$, 
namely\begin{align*}
 \int_0^\pi [\ee_\theta\cdot \nabla_x] [\ee_\theta\cdot \nabla_y]
\Bigl(\psi(\ee_\theta\cdot x)\psi(\ee_\theta\cdot y) \delta\bigl(\ee_\theta\cdot (x-y)\bigr)\Bigr) \,d\theta
 &=\int_0^\pi \frac{\psi(\ee_\theta\cdot x)\psi(\ee_\theta\cdot y)}{\bigl(\ee_\theta\cdot (x-y)\bigr)^2}\ \delta\bigl(\ee_\theta\cdot (x-y)\bigr)\,d\theta\\
&\qquad + \int_0^\pi \frac{\Psi(\theta,x,y)}{\bigl(\ee_\theta\cdot (x-y)\bigr)} \delta\bigl(\ee_\theta\cdot (x-y)\bigr)\,d\theta,
\end{align*}
where, for any $x \neq y$, $\Psi(\theta,x,y)$ is a smooth function of $\theta$ when $\ee_\theta$ is almost orthogonal to $(x-y)$
(that is, near the support of $\delta\bigl(\ee_\theta\cdot (x-y)\bigr)$). Hence, since $\delta\bigl(\ee_\theta\cdot (x-y)\bigr)$ is a distribution which is homogeneous of degree $-1$, we deduce that
$$
\biggl|\int_0^\pi \frac{\Psi(\theta,x,y)}{\bigl(\ee_\theta\cdot (x-y)\bigr)} \delta\bigl(\ee_\theta\cdot (x-y)\bigr)\,d\theta\biggr| \leq \frac{C}{|x-y|^2},
$$
and we obtain 
\begin{align*}
&\int_0^\pi \int_\R \psi(t)^2 |w_\theta'|^2(t)\,dt\,d\theta\\
&=\int_{\R^2} \int_{\R^2} |\var(x)-\var (y)|^2 \int_0^\pi \frac{\psi(\ee_\theta\cdot x)\psi(\ee_\theta\cdot y)}{\bigl(\ee_\theta\cdot (x-y)\bigr)^2}\ \delta\bigl(\ee_\theta\cdot (x-y)\bigr)\,d\theta\\
&\qquad+\int_{\R^2} \int_{\R^2} |\var(x)-\var (y)|^2 O(|x-y|^{-2})\,dx\,dy.
\end{align*}
We now observe that, being the expression inside the first integral  positive, it decreases if we localize it with a cut-off function $\chi(x)\chi(y)$.
In particular, if the support of $\chi(x)\chi(y)$ is contained inside the one of $\psi(\ee_\theta\cdot x)\psi(\ee_\theta\cdot y)$ for any $\theta \in [0,\pi]$,
since\footnote{
One way to prove this identity is to observe that $\int_0^\pi \frac{1}{\left(\ee_\theta\cdot (x-y)\right)^2}\ \delta\left(\ee_\theta\cdot (x-y)\right)\,d\theta$ is homogeneous of degree $-3$ in $x-y$,
it is invariant under rotations, and it is positive on positive functions.}
$$
\int_0^\pi \frac{1}{\bigl(\ee_\theta\cdot (x-y)\bigr)^2}\ \delta\bigl(\ee_\theta\cdot (x-y)\bigr)\,d\theta=\frac{\hat c}{|x-y|^3},\qquad \hat c>0,
$$
we get
\begin{equation}
\label{eq:local norm bound}
\begin{split}
\int_0^\pi \int_\R \psi(t)^2 |w_\theta'|^2(t)\,dt\,d\theta
&\geq \hat c\int_{\R^2} \int_{\R^2} \frac{|\var(x)-\var (y)|^2}{|x-y|^3} \chi(x)\chi(y)\,dx\,dy\\
&\qquad -C\int_{\R^2} \int_{\R^2} \frac{|\var(x)-\var (y)|^2}{|x-y|^2}\,dx\,dy
\end{split} 
\end{equation}
for any $\chi:\R^n\to [0,1]$ whose
support is contained inside the one of $\psi(\ee_\theta\cdot x)$ for any $\theta \in [0,\pi]$.

\section{Proof of Theorem \ref{thm:main}}
\label{sect:proof}
We first prove the result when $n=2$.

Let $\var_\e:=\one_E\ast \gamma_{2,\e}$, and let $w_\theta$ and $w_{\theta,\e}$ denote respectively the marginals of $\one_E$ and of $\var_\e$ (see \eqref{eq:w theta}).
Because Gaussian densities tensorizes, it is immediate to check that they commute with marginals, so the following identity holds:
$$
w_{\theta,\e}=w_\theta \ast \gamma_{1,\e}.
$$
By \eqref{eq:global norm bound} and Lemma \ref{lem:1} we see that 
$$
\frac{1}{|\log(\e)|}\int_0^\pi \int_\R |w_{\theta,\e}'|^2(t)\,dt\,d\theta=\frac{\bar c}{|\log(\e)|}\int_{\R^2} \int_{\R^2} \frac{|\var_\e(x)-\var_\e (y)|^2}{|x-y|^3}\,dx\,dy\leq C,
$$
so the measures 
$$
\nu_\e(dt,d\theta):=\frac{|w_{\theta,\e}'|^2(t)}{|\log(\e)|}\,dt\,d\theta,\qquad \mu_\e(dx,dy):=\frac{1}{|\log(\e)|}\frac{|\var_\e(x)-\var_\e (y)|^2}{|x-y|^3}\,dx\,dy
$$
are equibounded and they
weakly converge (up to a subsequence) to measures $\nu(dt,d\theta)$ and $\mu(dx,dy)$.

Since $\frac{|\var_\e(x)-\var_\e (y)|^2}{|\log(\e)|}\to 0$ as $\e \to 0$, it follows 
that the measure $\mu$ is concentrated on the diagonal $\{x=y\}$.

Concerning $\nu$, let us denote by $(a_\theta,b_\theta)$ the interval $\{w_\theta>0\}$, and observe that,
thanks to our assumption, there exists a constant $\bar C_\delta>0$ such that
$|w_{\theta,\e}'| \leq \bar C_\delta$ inside $[a_\theta+\delta,b_\theta-\delta]$ for all $\delta>0$, uniformly in $\e$ and $\theta$.
Hence, by integrating $\nu_\e$ against a test function $f(\theta,t)$ which is zero in a neighborhood of
$$
S:=\bigcup_{\theta \in [0,\pi]} \{a_\theta,b_\theta\}\times \{\theta\}\subset \R\times [0,\pi]
$$
and letting $\e \to 0$ we get
$$
\int f(t,\theta)\,\nu(dt,d\theta)=0,
$$
and by the arbitrariness of $f$ we deduce that
$\nu$ is concentrated on $S$.

We now make the following observation: in \eqref{eq:local norm bound} we have related the local $H^{1/2}$ norm of $\var_\e$
to another norm which depends only on the marginals of $\var_\e$.
The key fact is that the choice of the origin is completely arbitrary.
So, we argue as follows.  

Replace $E$ by $E^{(1)}$, the set of its density one points, i.e.,
$$
E^{(1)}:=\biggl\{x \in \R^n: \lim_{r \to 0}\frac{|E\cap B_r(x)|}{|B_r|}=1\biggr\}, 
$$
and define $\mathcal C$ as the convex hull of $E^{(1)}$.
If $E^{(1)}$ is not convex, then we can find a point $x_0 \in \partial^* E\setminus \partial \mathcal C$ which belongs to the interior of $\mathcal C$.
Let us fix a system of coordinates to that $x_0$ is the origin. With this choice,
for any $\theta \in [0,\pi]$ the set $(a_\theta,b_\theta)=\{w_\theta>0\}$ coincides with the projection of $\mathcal C$ onto the line in the direction of $\ee_\theta$,
and because $x_0$ belongs to the interior of $\mathcal C$
there exists a small constant $\rho>0$ such that
$$
[-\rho,\rho] \subset (a_\theta,b_\theta)\qquad \forall\,\theta\in [0,\pi].
$$
Let us take a cut-off function $\psi(t)$ supported inside $[-\rho,\rho]$,
and then choose a cut-off function $0 \leq \chi \leq 1$ such that 
the support of $\chi(x)\chi(y)$ is contained inside the one of $\psi(\ee_\theta\cdot x)\psi(\ee_\theta\cdot y)$  for any $\theta \in [0,\pi]$,
and $\chi =1$ inside $B_{r_0}(x_0)$ for some small $r_0$. Then, since
$\nu$ is concentrated on $S$, by \eqref{eq:local norm bound} applied to $w_{\theta,\e}$ and $\var_\e$, Lemma \ref{lem:2}, \eqref{eq:fullmeasure}, and the fact that $\mu$ is concentrated on $\{x=y\}$, we get
\begin{align*}
0&= \int_0^\pi \int_\R \psi(t)^2\nu(dt,d\theta)\\
&\geq\liminf_{\e\to 0} \frac{\hat c}{|\log(\e)|}\int_{B_{r_0}(x_0)} \int_{B_{r_0}(x_0)} \frac{|\var_\e(x)-\var_\e (y)|^2}{|x-y|^3} \,dx\,dy \\
& \qquad \qquad \qquad 
\qquad \qquad \qquad -\frac{C}{|\log(\e)|}\int_{\R^2} \int_{\R^2} \frac{|\var_\e(x)-\var_\e(y)|^2}{|x-y|^2}\,dx\,dy \\
&\gtrsim r_0^{n-1} -C\int_{\R^2} \int_{\R^2}|x-y|\,\mu(dx,dy) = r_0^{n-1}>0
\end{align*}
a contradiction which concludes the proof in the two dimensional case.\\

For the general case we argue as follows: let $\pi\subset \R^n$ be a two dimensional plane passing through the origin,
and for any $\ee \in \mathbb S^{n-1}\cap \pi$ consider the projection of $E$ onto the hyperplane $\ee^\perp$.

Thanks to our assumption, if we slice $E$ with some translate $\pi_v:=\pi+v$ ($v \in \R^n$) of $\pi$,
by the slicing formula (see for instance \cite[Theorem 18.11 and Remark 18.13]{M}) 
the set $E\cap \pi_v\subset \pi_v\simeq \R^2$ is a bounded set of finite perimeter for $\H^n$-a.e. $v \in \R^n$. In addition,
$E\cap \pi_v$ satisfies the assumptions of our theorem with $n=2$.
Hence, by what we proved above, $E\cap \pi_v$ coincides with a convex set up to a set of $\H^2$-measure zero.

We now show that $E^{(1)}$ is convex.
Fix $x,y \in E^{(1)}$ and $t \in (0,1)$, and pick a plane $\pi$ such that $x-y \in \pi$. By the discussion above
we deduce that, for $\H^n$-a.e. $v \in B_r(y)$,
$$
E\cap \pi_v \quad \text{is equal to a convex set up to a set of $\H^2$-measure zero}.
$$
Hence, by Fubini's Theorem we obtain the following:
for $\H^n$-a.e. $v \in B_r(y)$,
for $\H^2$-a.e. $z \in B_r(y)\cap E \cap (E-(x-y))\cap \pi_v$, the set
$$
[z,z+(x-y)]\cap E\quad \text{is equal to a segment up to a set of $\H^1$-measure zero}
$$
(here we use $[z,z+(x-y)]$ to denote the segment from $z$ to $z+(x-y)$).
From this fact and  Fubini's theorem (again), it follows that
\begin{multline}
\label{eq:segment}
\H^1\bigl(\bigl([z,z+(x-y)]\cap B_r(tx+(1-t)y)\bigr)\setminus E\bigr)
=0\qquad \text{for $\H^n$-a.e. $z \in B_r(y)\cap E \cap (E-(x-y))$.}
\end{multline}
Since both $x,y \in E^{(1)}$ we see that
$$
\liminf_{r\to 0}\frac{|B_r(y)\cap E \cap (E-(x-y))|}{|B_r|} \geq 1 - \lim_{r\to 0}\frac{|B_r(y)\setminus E|}{|B_r|} - \lim_{r\to 0}\frac{|B_r(x)\setminus E|}{|B_r|}=1,
$$
so it follows easily from \eqref{eq:segment} that 
$$
\lim_{r\to 0}\frac{|B_r(tx+(1-t)y)\setminus E|}{|B_r|} =0,
$$
proving that $tx+(1-t)y \in E^{(1)}$. By the arbitrariness of $x,y,t$ we deduce that $E^{(1)}$ is convex, concluding the proof.

\end{document}